\documentclass{amsart}
\usepackage{amscd,amsthm,amsmath,amssymb,mathrsfs}
\usepackage{vmargin}
\usepackage{hyperref}
\usepackage{enumerate}
\usepackage{comment}

\newtheorem{theorem}{Theorem}[section]
\newtheorem{proposition}[theorem]{Proposition}
\newtheorem{lemma}[theorem]{Lemma}

\newtheorem{remark}[theorem]{Remark}

\setpapersize{A4}
\setmargins{2.5cm}      
{1.5cm}                        
{16.5cm}                      
{23.42cm}                   
{11pt}                          
{1cm}                           
{0pt}                          
{2cm} 

\def\diam{\mathrm{diam}}
\def\Lipo{\mathrm{Lip}_0}

\def\Lip{\mathrm{Lip}}

\begin{document}

\title{Isometric composition operators on Lipschitz spaces}
\author{A. Jim\'{e}nez-Vargas}
\address{Departamento de Matem\'{a}ticas, Universidad de Almer\'{i}a, 04120 Almer\'{i}a, Spain}
\email{ajimenez@ual.es}
\date{Accepted for publication in Mediterranean Journal of Mathematics and it will appear in 2020.}

\subjclass[2010]{47B33, 47B38}
%47B33 Composition operators
%47B38 Operators on function spaces (general)
\keywords{Composition operator; Lipschitz function; isometry; peak property.}

\begin{abstract}
Given pointed metric spaces $X$ and $Y$, we characterize the basepoint-preserving Lipschitz maps $\phi$ from $Y$ to $X$ inducing an isometric composition operator $C_\phi$ between the Lipschitz spaces $\Lipo(X)$ and $\Lipo(Y)$, whenever $X$ enjoys the peak property. This gives an answer to a question posed by N. Weaver in his book [Lipschitz algebras. Second edition. World Scientific Publishing Co. Pte. Ltd., Hackensack, NJ, 2018]. 
\end{abstract}
\maketitle

\section{Introduction}

Let $(X,d)$ be a pointed metric space with a basepoint designated by $e_X$, and let $\widetilde{X}$ denote the set 
$$
\left\{(x,y)\in X\times X\colon x\neq y\right\}.
$$
The Lipschitz space $\Lipo(X)$ is the Banach space of all Lipschitz functions $f\colon X\to\mathbb{R}$ with $f(e_X)=0$, under the Lipschitz norm: 
$$
\Lip(f)=\sup\left\{\frac{\left|f(x)-f(y)\right|}{d(x,y)}\colon (x,y)\in\widetilde{X}\right\}.
$$
Throughout the paper, unless specified otherwise, $X$ and $Y$ will denote two pointed metric spaces. Every Lipschitz map $\phi$ from $Y$ to $X$ which preserves the basepoint produces a bounded composition operator $C_\phi$ from $\Lipo(X)$ to $\Lipo(Y)$, defined by $C_\phi f=f\circ\phi$ for all $f\in\Lipo(X)$. The map $\phi$ is known as the symbol of the operator $C_\phi$. 

The problem of characterizing those symbols $\phi$ which induce isometric composition operators $C_\phi$ (not necessarily surjective) has been raised recently by N. Weaver in \cite[p. 53]{Wea-18}. The same question was addressed by some authors for isometric composition operators on Banach spaces of analytic functions (see \cite{MarVuk-07} and the papers that cite it). 

In the surjective case, N. Weaver proved in \cite[Proposition 2.28 (iii)]{Wea-18} that those basepoint-preserving Lipschitz symbols $\phi$ from $Y$ to $X$ that generate surjective isometric composition operators $C_\phi$ from $\Lipo(X)$ to $\Lipo(Y)$ are precisely the surjective isometries from $Y$ to $X$, whenever $X$ and $Y$ are complete.

A description of all linear isometries (surjective or not) of $\Lipo(X)$ does not seem to be known. Given $\alpha\in ]0,1[$, we denote by $X^\alpha$ the metric space $(X,d^\alpha)$. If $X$ and $Y$ are compact, the linear isometries from $\Lipo(X^\alpha)$ onto $\Lipo(Y^\alpha)$ were characterized by E. Mayer-Wolf in \cite[Theorem 3.3]{May-81}. He showed that a linear operator $T\colon\Lipo(X^\alpha)\to\Lipo(Y^\alpha)$ is a surjective isometry if and only if it is of the form 
$$
T(f)(y)=\lambda k^{-\alpha}\left(f(\phi(y))-f(\phi(e_Y))\right)
$$
for all $f\in\Lipo(X^\alpha)$ and $y\in Y$, where $\lambda\in\mathbb{R}$ with $|\lambda|=1$ and $\phi\colon Y\to X$ is a bijective $k$-dilation with $\diam(X)=k\cdot\diam(Y)$. Given $k>0$, a map $\phi\colon Y\to X$ is a $k$-dilation if $d_X(\phi(x),\phi(y))=k\cdot d_Y(x,y)$ for all $x,y\in Y$. Mayer-Wolf's result was extended by N. Weaver for surjective linear isometries from $\Lipo(X)$ to $\Lipo(Y)$, when $X$ and $Y$ are complete and uniformly concave \cite[Theorem 3.56]{Wea-18}.  

According to \cite[Definition 3.33]{Wea-18}, a metric space $X$ is said to be concave if 
$$
d(x,y)<d(x,z)+d(z,y)
$$
for any triple of distinct points $x,y,z\in X$, and uniformly concave if for every distinct points $x,y\in X$ and every $\varepsilon>0$, there exists $\delta>0$ such that 
$$
d(x,y)\leq d(x,z)+d(z,y)-\delta
$$
for all $z\in X$ such that $d(x,z),d(y,z)\geq\varepsilon$. The class of uniformly concave metric spaces includes any closed subset of $\mathbb{R}^n$ with the Euclidean norm or any compact subset of a strictly convex Banach space both without colinear triples, the unit sphere of any uniformly convex Banach space and H\"older spaces, among others (see \cite[Section 3.5]{Wea-18}). %The unit sphere of any strictly convex Banach space $X$ is concave, but in general it is not uniformly concave.

Uniform concavity is closely related to the extremal structure of the closed unit ball $B_{\mathcal{F}(X)}$ of the Lipschitz free Banach space 
$$
\mathcal{F}(X):=\overline{\mathrm{lin}}\left\{\delta_x\colon x\in X\right\}\subset\Lipo(X)^*,
$$
where $\delta_x(f):=f(x)$ for every $x\in X$ and $f\in\Lipo(X)$. Let us recall that $\mathcal{F}(X)$ is the canonical predual of $\Lipo(X)$. By Theorems 3.39 in \cite{Wea-18} and 4.1 in \cite{AliGui-19}, $X$ is uniformly concave if and only if every molecule $(\delta_x-\delta_y)/d(x,y)$ is a preserved extreme point of $B_{\mathcal{F}(X)}$. 

The notion of peaking function has been a very important tool in the study of the isometric theory of Lipschitz spaces. According to \cite[Definition 2.4.1]{Wea-99}, a function $f\in\Lipo(X)$ with $\Lip(f)\leq 1$ is said to peak at $(x,y)\in\widetilde{X}$ if 
$$
\frac{f(x)-f(y)}{d(x,y)}=1,
$$
and for each open set $U\subset\widetilde{X}$ containing $(x,y)$ and $(y,x)$, there exists $\delta>0$ such that 
$$
\frac{\left|f(z)-f(w)\right|}{d(z,w)}\leq 1-\delta
$$
for all $(z,w)\in\widetilde{X}\setminus U$. Colloquially, if $f$ peaks at $(x,y)\in\widetilde{X}$, we have that $|f(z)-f(w)|/d(z,w)$ is uniformly less than 1 when $(z,w)$ is away from $(x,y)$ and $(y,x)$.

We say that a pointed metric space $X$ has the peak property if for every $(x,y)\in\widetilde{X}$, there is a function $f\in\Lipo(X)$ with $\Lip(f)\leq 1$ that peaks at $(x,y)$. H\"older spaces constitute a class of metric spaces with the peak property (see the proof of Proposition 2.4.5 in \cite{Wea-99}). By \cite[Theorem 5.4]{GarProRue-18}, $X$ has the peak property if and only if every molecule $(\delta_x-\delta_y)/d(x,y)$ is a strongly exposed point of $B_{\mathcal{F}(X)}$. 

In this note, we characterize all basepoint-preserving Lipschitz maps $\phi$ from $Y$ to $X$ whose induced composition operators $C_\phi$ from $\Lipo(X)$ to $\Lipo(Y)$ are isometries, whenever $X$ has the peak property. We also give a condition for $\phi$ to induce an isometric composition operator $C_\phi$ without any restriction on $X$.

\section{The results}

Let us recall that a map $\phi\colon Y\to X$ is nonexpansive if $d_X(\phi(x),\phi(y))\leq d_Y(x,y)$ for all $x,y\in Y$. 
%, that is, $\Lip(\phi)\leq 1$ 
A nonexpansive map $\phi\colon Y\to X$ which preserves the basepoint can induce or not an isometric composition operator $C_\phi\colon\Lipo(X)\to\Lipo(Y)$. 
For example, each $k$-dilation $\phi\colon Y\to X$ with $k\in ]0,1]$ is nonexpansive and if, in addition, $\phi\colon Y\to X$ is nonconstant, preserves the basepoint and has dense range, then $C_\phi\colon\Lipo(X)\to\Lipo(Y)$ is an isometry if and only if $k=1$ (that is, if $\phi$ is an isometry). 
%Note that if $\phi$ is (only) constant, then $Y=\{e_Y\}$, hence $\Lipo(Y)=\{0\}$, $C_\phi=\{0\}$ and so $C_\phi$ is isometric for any $k$. If $\phi$ is nonconstant, then there exist at least two distinct points in $Y$ (also in $X$) and we can obtain that $\Lip(f\circ\phi)=k\Lip(f)$ for every $f\in\Lipo(X)$.

We first give a sufficient condition for a basepoint-preserving Lipschitz map $\phi$ from $Y$ to $X$ to be the symbol of an isometric composition operator from $\Lipo(X)$ to $\Lipo(Y)$.

\begin{theorem}\label{t11}
Let $X$ and $Y$ be pointed metric spaces and let $\phi\colon Y\to X$ be a Lipschitz map which preserves the basepoint. Assume that $\phi$ is nonexpansive and satisfies the property (M): for every point $(x,y)\in\widetilde{X}$, there exists a sequence $\{(x_n,y_n)\}$ in $\widetilde{Y}$ such that $\{\phi(x_n)\}\to x$, $\{\phi(y_n)\}\to y$ and  
$$
\left\{\frac{d_X(\phi(x_n),\phi(y_n))}{d_Y(x_n,y_n)}\right\}\to 1.
$$
Then $C_\phi\colon\Lipo(X)\to\Lipo(Y)$ is an isometry.
\end{theorem}

\begin{proof}
Since $\phi$ is nonexpansive, we have 
$$
\Lip(f\circ\phi)\leq\Lip(f)\Lip(\phi)\leq\Lip(f)
$$
for every $f\in\Lipo(X)$. In order to check the converse inequality, take $f\in\Lipo(X)$. Hence there exists a sequence $\{(a_m,b_m)\}$ in $\widetilde{X}$ such that  
$$
\lim_{m\to\infty}\frac{\left|f(a_m)-f(b_m)\right|}{d_X(a_m,b_m)}=\Lip(f).
$$
Fix $m\in\mathbb{N}$. By assumption we can find a sequence $\{(x_n^{(m)},y_n^{(m)})\}$ in $\widetilde{Y}$ satisfying that 
$$
\lim_{n\to\infty}d_X(\phi(x_n^{(m)}),a_m)=0,\quad \lim_{n\to\infty}d_X(\phi(y_n^{(m)}),b_m)=0\quad% \text{and}\quad 
$$
and 
$$
\lim_{n\to\infty}\frac{d_X(\phi(x_n^{(m)}),\phi(y_n^{(m)}))}{d_Y(x_n^{(m)},y_n^{(m)})}=1.
$$
It follows that $\lim_{n\to\infty}d_X(\phi(x_n^{(m)}),\phi(y_n^{(m)}))=d_X(a_m,b_m)>0$ and therefore there exists $p\in\mathbb{N}$ such that $d_X(\phi(x_n^{(m)}),\phi(y_n^{(m)}))>0$ for all $n\geq p$. We have  
\begin{align*}
\Lip(f\circ\phi)
&=\sup_{x\neq y}\frac{\left|f(\phi(x))-f(\phi(y))\right|}{d_Y(x,y)}\\
&\geq\frac{\left|f(\phi(x_n^{(m)}))-f(\phi(y_n^{(m)}))\right|}{d_Y(x_n^{(m)},y_n^{(m)})}\\
&=\frac{\left|f(\phi(x_n^{(m)}))-f(\phi(y_n^{(m)}))\right|}{d_X(\phi(x_n^{(m)}),\phi(y_n^{(m)}))}\frac{d_X(\phi(x_n^{(m)}),\phi(y_n^{(m)}))}{d_Y(x_n^{(m)},y_n^{(m)})}\\
&\geq\frac{d_X(\phi(x_n^{(m)}),\phi(y_n^{(m)}))}{d_Y(x_n^{(m)},y_n^{(m)})}\left(\frac{\left|f(a_m)-f(b_m)\right|}{d_X(\phi(x_n^{(m)}),\phi(y_n^{(m)}))}\right.\\
&\left.-\Lip(f)\frac{d_X(a_m,\phi(x_n^{(m)}))}{d_X(\phi(x_n^{(m)}),\phi(y_n^{(m)}))}
-\Lip(f)\frac{d_X(b_m,\phi(y_n^{(m)}))}{d_X(\phi(x_n^{(m)}),\phi(y_n^{(m)}))}\right)
\end{align*}
for all $n\geq p$, and taking limits as $n\to\infty$, we obtain 
$$
\Lip(f\circ\phi)\geq \frac{\left|f(a_m)-f(b_m)\right|}{d_X(a_m,b_m)}.
$$
Since $m$ was arbitrary, we conclude that
$$
\Lip(f\circ\phi)\geq \lim_{m\to\infty}\frac{\left|f(a_m)-f(b_m)\right|}{d_X(a_m,b_m)}=\Lip(f).
$$
\end{proof} 

Note that there are symbols $\phi$ satisfying the conditions in Theorem \ref{t11}; for example, every basepoint-preserving isometry with dense range $\phi\colon Y\to X$.% is nonexpansive and enjoys the property (M). 

In general, we can establish a kind of reciprocal result of Theorem \ref{t11}.

\begin{proposition}\label{prop-inicial}
Let $X$ and $Y$ be pointed metric spaces and let $\phi\colon Y\to X$ be a Lipschitz map which preserves the basepoint. Assume that $C_\phi\colon\Lipo(X)\to\Lipo(Y)$ is an isometry. Then $\phi$ is nonexpansive and has the following additional property: for every point $(x,y)\in\widetilde{X}$, there exists a sequence $\{(x_n,y_n)\}$ in $\widetilde{Y}$ such that $\{\phi(x_n)\}\to x$, $\{\phi(y_n)\}\to y$ and 
%$$\lim_{n\to\infty}d_X(\phi(x_n),x)=0,\quad \lim_{n\to\infty}d_X(\phi(y_n),y)=0$$ 
$$
\lim_{n\to\infty}\frac{d_X(\phi(x_n),\phi(y_n))}{d_Y(x_n,y_n)}\leq 1.
$$
\end{proposition}

\begin{proof} 
Clearly, $\left\|C_\phi\right\|\leq 1$, and since $\left\|C_\phi\right\|=\Lip(\phi)$ by \cite[Proposition 2.23]{Wea-18} (completeness of $X$ and $Y$ is not needed to prove this formula in \cite{Wea-18}), it follows that $\phi$ is nonexpansive. In order to show that $\phi$ has the above-cited property, let $(x,y)\in\widetilde{X}$. Note that $C_\phi$ is injective and therefore $\phi(Y)$ is dense in $X$ by \cite[Proposition 2.25 (ii)]{Wea-18} (completeness of $X$ and $Y$ is not necessary to prove this fact). Hence we can take sequences $\{x_n\}$ and $\{y_n\}$ in $Y$ such that $\{\phi(x_n)\}\to x$ and $\{\phi(y_n)\}\to y$. It follows that $\lim_{n\to\infty}d_X(\phi(x_n),\phi(y_n))=d_X(x,y)>0$, hence there exists $p\in\mathbb{N}$ such that $d_X(\phi(x_n),\phi(y_n))>0$ for all $n\geq p$ and thus $d_Y(x_n,y_n)>0$ for all $n\geq p$. Since $d_X(\phi(x_n),\phi(y_n))/d_Y(x_n,y_n)\leq 1$ for all $n\geq p$, taking subsequences if necessary, we obtain that
$$
\lim_{n\to\infty}\frac{d_X(\phi(x_n),\phi(y_n))}{d_Y(x_n,y_n)}\leq 1.
$$
\end{proof} 

We shall next prove that the basepoint-preserving Lipschitz maps $\phi\colon Y\to X$ for which $C_\phi$ is an isometry from $\Lipo(X)$ to $\Lipo(Y)$, are precisely the nonexpansive maps satisfying the property (M), whenever $X$ has the peak property. 

We shall make use of the following sequential characterization of peaking functions. It appears without proof in \cite{GarProRue-18} and we prove it here for completeness. 

\begin{lemma}\label{lema1}\cite{GarProRue-18}
Let $X$ be a pointed metric space, $(x,y)\in\widetilde{X}$ and $f\in\Lipo(X)$ with $\Lip(f)\leq 1$. Then $f$ peaks at $(x,y)$ if and only if 
$$
\frac{f(x)-f(y)}{d(x,y)}=1,
$$
and the following property (P) holds: 
if $\{(x_n,y_n)\}$ is a sequence in $\widetilde{X}$ such that 
$$
\left\{\frac{f(x_n)-f(y_n)}{d(x_n,y_n)}\right\}\to 1,
$$
then $\{x_n\}\to x$ and $\{y_n\}\to y$.
\end{lemma}

\begin{proof}
Assume that $f$ peaks at $(x,y)$. Then 
$$
\frac{f(x)-f(y)}{d(x,y)}=1.
$$
In order to prove that $f$ satisfies the property (P), let $\{(x_n,y_n)\}$ be a sequence in $\widetilde{X}$ such that 
$$
\left\{\frac{f(x_n)-f(y_n)}{d(x_n,y_n)}\right\}\to 1.
$$
If the conclusion of the property (P) is not satisfied, we could find a real number $\varepsilon>0$ and subsequences $\{x_{\sigma(n)}\}$ and $\{y_{\sigma(n)}\}$ of $\{x_n\}$ and $\{y_n\}$, respectively, satisfying that $d(x_{\sigma(n)},x)\geq\varepsilon$ for all $n\in\mathbb{N}$ or $d(y_{\sigma(n)},y)\geq\varepsilon$ for all $n\in\mathbb{N}$. Clearly, the set 
$$
\left\{n\in\mathbb{N}\colon d(x_{\sigma(n)},y)\geq\varepsilon\right\}\cup\left\{n\in\mathbb{N}\colon d(y_{\sigma(n)},x)\geq\varepsilon\right\}
$$
is nonempty. %In contrary case we would have that $\{x_{\sigma(n)}\}\to y$ and $\{y_{\sigma(n)}\}\to x$ and therefore
%$$\lim_{n\to +\infty}\frac{f(x_{\sigma(n)})-f(y_{\sigma(n)})}{d(x_{\sigma(n)},y_{\sigma(n)})}=-1,$$ which is impossible. 
Taking subsequences of $\{x_{\sigma(n)}\}$ and $\{y_{\sigma(n)}\}$, we can suppose that $d(x_{\sigma(n)},y)\geq\varepsilon$ for all $n\in\mathbb{N}$ or $d(y_{\sigma(n)},x)\geq\varepsilon$ for all $n\in\mathbb{N}$. Since $f$ peaks at $(x,y)$, there exists $\delta>0$ such that 
$$
\frac{\left|f(x_{\sigma(n)})-f(y_{\sigma(n)})\right|}{d(x_{\sigma(n)},y_{\sigma(n)})}\leq 1-\delta
$$
for all $n\in\mathbb{N}$, and since 
$$
\left\{\frac{f(x_n)-f(y_n)}{d(x_n,y_n)}\right\}\to 1,
$$
we would arrive at a contradiction. This proves that $\{x_n\}\to x$ and $\{y_n\}\to y$.
 
Conversely, suppose that 
$$
\frac{f(x)-f(y)}{d(x,y)}=1
$$
and the property (P) is satisfied, but $f$ does not peak at $(x,y)$. Hence there exist $\varepsilon>0$ and a sequence $\{(x_n,y_n)\}$ in $\widetilde{X}$ satisfying that $\max\left\{d(x_n,x),d(y_n,y)\right\}\geq\varepsilon$ for all $n\in\mathbb{N}$ and $\max\left\{d(x_n,y),d(y_n,x)\right\}\geq\varepsilon$ for all $n\in\mathbb{N}$ such that 
$$
\frac{f(x_n)-f(y_n)}{d(x_n,y_n)}>1-\frac{1}{n}
$$
for all $n\in\mathbb{N}$. Since $\Lip(f)\leq 1$, it follows that  
$$
\left\{\frac{f(x_n)-f(y_n)}{d(x_n,y_n)}\right\}\to 1,
$$
but the sequence $\{(x_n,y_n)\}$ does not satisfy the conclusion of the property (P), a contradiction.
\end{proof}

We are now ready to prove our main result.

\begin{theorem}\label{t12}
Let $X$ and $Y$ be pointed metric spaces and let $\phi\colon Y\to X$ be a basepoint-preserving Lipschitz map. Assume that $X$ enjoys the peak property. Then $C_\phi\colon\Lipo(X)\to\Lipo(Y)$ is an isometry if and only if $\phi$ is nonexpansive and satisfies the property (M): for every point $(x,y)\in\widetilde{X}$, there exists a sequence $\{(x_n,y_n)\}$ in $\widetilde{Y}$ such that $\{\phi(x_n)\}\to x$, $\{\phi(y_n)\}\to y$ and  
$$
\left\{\frac{d_X(\phi(x_n),\phi(y_n))}{d_Y(x_n,y_n)}\right\}\to 1.
$$
\end{theorem}

\begin{proof}
The sufficiency follows from Theorem \ref{t11}. To prove the necessity, assume that $C_\phi\colon\Lipo(X)\to\Lipo(Y)$ is an isometry. Then $\phi$ is nonexpansive by Proposition \ref{prop-inicial}. We now show that $\phi$ enjoys the property (M). Let $(x,y)\in\widetilde{X}$. Since $X$ has the peak property, there exists a function $f_{(x,y)}\in\Lipo(X)$ with $\Lip(f_{(x,y)})=1$ that peaks at $(x,y)$. Note that $\Lip(f\circ\phi)=\Lip(f)$ for all $f\in\Lipo(X)$. Since $\Lip(f_{(x,y)}\circ\phi)=1$, we can take a sequence $\{(x_n,y_n)\}$ in $\widetilde{Y}$ such that 
$$
\left\{\frac{\left|f_{(x,y)}(\phi(x_n))-f_{(x,y)}(\phi(y_n))\right|}{d_Y(x_n,y_n)}\right\}\to 1.
$$
Using that $\Lip(f_{(x,y)})=1$ and $\Lip(\phi)\leq 1$, we obtain 
$$
\left\{\frac{d_X(\phi(x_n),\phi(y_n))}{d_Y(x_n,y_n)}\right\}\to 1.
$$
An easy argument yields that 
$$
\left\{\frac{\left|f_{(x,y)}(\phi(x_n))-f_{(x,y)}(\phi(y_n))\right|}{d_X(\phi(x_n),\phi(y_n))}\right\}\to 1.
$$
Taking subsequences, we have that 
$$
\left\{\frac{f_{(x,y)}(\phi(x_{\sigma(n)})-f_{(x,y)}(\phi(y_{\sigma(n)}))}{d_X(\phi(x_{\sigma(n)}),\phi(y_{\sigma(n)}))}\right\}\to 1,
$$
or 
$$
\left\{\frac{f_{(x,y)}(\phi(x_{\sigma(n)}))-f_{(x,y)}(\phi(y_{\sigma(n)}))}{d_X(\phi(x_{\sigma(n)}),\phi(y_{\sigma(n)}))}\right\}\to -1.
$$
Applying Lemma \ref{lema1}, it follows that $\{\phi(x_{\sigma(n)})\}\to x$ and $\{\phi(y_{\sigma(n)})\}\to y$ in the first case, or $\{\phi(y_{\sigma(n)})\}\to x$ and $\{\phi(x_{\sigma(n)})\}\to y$ in the second one. This proves the theorem.
\end{proof} 

\begin{remark}
Our description of isometric composition operators on $\Lipo(X)$ is a Lipschitz version of a characterization of isometric composition operators on the Bloch space $\mathcal{B}$, obtained by M. Mart\'{i}n and D. Vukoti\'c \cite{MarVuk-07}. 

In view of another characterization of isometric composition operators on $\mathcal{B}$ stated by F. Colonna in \cite[Theorem 5]{Col-05}, it would be interesting to study under which conditions, the basepoint-preserving Lipschitz self-maps $\phi$ of $X$ inducing an isometric composition operator $C_\phi$ on $\Lipo(X)$ are precisely those having Lipschitz constant equal to one.
\end{remark} 

\textbf{Acknowledgements.} This research was partially supported by Junta de Andaluc\'{\i}a grant FQM194 and project UAL18-FQM-B018-A. We would like to thank Abraham Rueda for letting us know about Lemma \ref{lema1} and for us supplying with a copy of his papers. This note was written during the review \cite{Jim-19} of the monograph \cite{Wea-18} for Mathematical Reviews/MathSciNet.

\bibliographystyle{amsplain}

\begin{thebibliography}{99}
\bibitem{Jim-19} A. Jim\'enez-Vargas, MR3792558 (this is the review of \cite{Wea-18}).
\end{thebibliography}

\end{document}